\documentclass[11pt,a4paper]{amsart}
\usepackage{amssymb}
\usepackage[all]{xy}

\theoremstyle{plain}
\newtheorem{theorem}{Theorem}[section]
\newtheorem{lemma}[theorem]{Lemma}
\newtheorem{corollary}[theorem]{Corollary}
\theoremstyle{remark}
\newtheorem{remark}[theorem]{Remark}
\newtheorem{example}[theorem]{Example}

\newcommand{\C}{\mathbb{C}}
\newcommand{\Z}{\mathbb{Z}}
\newcommand{\Fq}{\mathbb{F}_{q}}
\DeclareMathOperator{\charact}{char}
\DeclareMathOperator{\spec}{Spec}
\DeclareMathOperator{\fix}{Fix}

\begin{document}
\title[Modular quotient singularities]{Three-dimensional isolated quotient \\
singularities in odd characteristic}
\author{D.~A. Stepanov}
\address{The Department of Mathematical Modelling \\
         Bauman Moscow State Technical University \\
         ul. 2-ya Baumanskaya 5, Moscow 105005, Russia}
\email{dstepanov@bmstu.ru}
\thanks{The research was supported by the Russian Grant for Leading Scientific 
Schools, grant no. 5139.2012.1, and RFBR, grant no. 11-01-00336-a.}
\date{}
\begin{abstract}
Let a finite group $G$ act linearly on a finite dimensional vector space $V$
over an algebraically closed field $k$ of characteristic $p>2$. Assume that
the quotient $V/G$ is an isolated singularity. In the case when $p$ does not
divide the order of $G$, isolated singularities $V/G$ are completely classified
and their classification reduces to Zassenhaus-Vincent-Wolf classification of
isolated quotient singularities over $\C$. In the present paper we show that
if $\dim V=3$, then also in the modular case $p\,|\,|G|$ classification of
isolated quotient singularities reduces to Zassenhaus-Vincent-Wolf 
classification. Some remarks on modular quotient singularities in other
dimensions and in even characteristic are also given.
\end{abstract}
\maketitle

\section{Introduction}
Let $V$ be an algebraic variety defined over an algebraically closed field
$k$ of characteristic $p$. Let $G$ be a finite group acting on $V$ and $P\in V$
a fixed closed point of this action. Under these assumptions there is a 
well-defined quotient algebraic variety $V/G$. Denote by $\pi\colon V\to V/G$
the natural projection. We say that $Q=\pi(P)\in V/G$ is an \emph{isolated
singularity}, if the variety $V/G$ is singular at $Q$ (i.e., the local ring
$\mathcal{O}_{V/G,Q}$ is not regular), and there are no other singular points
of $V/G$ in some Zariski open neighborhood of $Q$. If characteristic $p$
does not divide the order $|G|$ of $G$, isolated quotient singularities are 
completely classified up to formal or, when $k=\C$ is the field of complex 
numbers, up to analytic equivalence. The classification over $\C$ was obtained
by H. Zassenhaus, G. Vincent, and J. Wolf as a part of classification of
manifolds of constant positive curvature. It is summarized in Wolf's book
\cite{Wolf}, Chapters~5, 6, 7 (see also our survey in \cite{St}, which is 
written from a point of view of singularity theory). It is not hard to see 
that Zassenhaus-Vincent-Wolf classification covers also isolated quotient
singularities over arbitrary algebraically closed field $k$ of characteristic 
$0$ and of prime characteristic $p$ where $p$ does not divide $|G|$.
The modular case $p\,|\,|G|$ remains open.

The first difficulty in the modular case is that the action of $G$ on $V$
is not in general linearizable in a formal neighborhood of a fixed point $P$.
This means that it is in general impossible to choose a system of parameters 
$x_1,\ldots,x_n$ in the complete local ring $\widehat{\mathcal{O}_{V,P}}$ so
that $G$ acts on $\widehat{\mathcal{O}_{V,P}}$ via linear substitutions in
$x_1,\ldots,x_n$. Still a lot of isolated quotient singularities arise as 
quotients by a nonlinear action of a modular group, see, e.g., \cite{Lorenzini}. 
Thus the classification problem does not reduce to a problem of linear 
representation theory as in the nonmodular case.  Nonlinear actions are 
not considered in this paper, so from now on we assume that $V$ is a 
vector space and $G$ is a subgroup of the general linear group $GL(V)$. 
The second difficulty lies in the fact that the inverse 
Chevalley-Shephard-Todd Theorem does not hold in the modular case. A linear 
map $g\colon V\to V$ is called a \emph{pseudoreflection} if it has finite 
order and the set of points fixed by $g$ is a hyperplane. By $S(V^*)$ we 
denote the symmetric algebra of the dual space $V^*$ of $V$.
\begin{theorem}[Chevalley-Shephard-Todd Theorem, \cite{Benson}, Theorem~7.2.1]
\label{T:CST}
Let V be a finite dimensional vector space over a field $k$ of characteristic
$p$, $p\geq 0$, and $G$ a subgroup of $GL(V)$. Then if the ring of invariants 
$S(V^*)^G$ of $G$ is polynomial (and in particular $V/G$ is nonsingular), then 
the group $G$ is generated by pseudoreflections. If the characteristic $p$ does
not divide the order of the group $G$, then the converse also holds.
\end{theorem}
In the modular case, a quotient $V/G$ may be singular even for $G$ generated by
pseudoreflections, as is shown by the example of the symmetric group $S_6$
in its $4$-dimensional irreducible $2$-modular representation obtained as the
nontrivial constituent of the permutation module, see \cite{KM}, Example~2.2.
One can check that in this example the quotient singularity $V/S_6$ is not 
isolated. In fact, this is a general phenomenon as follows from a remarkable 
result of G. Kemper and G. Malle:
\begin{theorem}[\cite{KM}, Main Theorem]\label{T:KM}
Let $V$ be a finite dimensional vector space and $G$ a finite irreducible 
subgroup of $GL(V)$. Then $S(V^*)^G$ is a polynomial ring if and only if $G$ 
is generated by pseudoreflections and the pointwise stabilizer in $G$ of any
nontrivial subspace of $V$ has a polynomial ring of invariants.
\end{theorem}
If the condition ``irreducible'' could be eliminated from the statement of 
Theorem~\ref{T:KM}, the classification of isolated (linear) quotient 
singularities would reduce to the nonmodular case, that is, to
Zassenhaus-Vincent-Wolf classification. Moreover, in Section~\ref{S:KMandIQS} 
we prove the following equivalence.
\begin{theorem}\label{T:equiv}
The following two statements are equivalent:
\begin{enumerate}
\item Kemper-Malle Theorem~\ref{T:KM} holds for all, not only irreducible,
finite linear groups $G$;
\item let $G$ be a finite subgroup of $GL(V)$ and $H$ the normal subgroup of $G$
generated by pseudoreflections. If $V/G$ is an isolated singularity, then
$V/H$ is nonsingular, $p(=\charact k) \nmid |G/H|$, and 
$$V/G \simeq (V/H)/(G/H),$$
i.e., any isolated quotient singularity $V/G$ is naturally isomorphic to a 
nonmodular quotient singularity.
\end{enumerate}
\end{theorem}
It should be noted that the induced action of $G/H$ on $V/H$ may be nonlinear,
but since the group $G/H$ is nonmodular, its action can be locally formally
linearized in a neighborhood of the fixed point.

In Kemper-Malle Theorem~\ref{T:KM}, the ground field $k$ need not be 
algebraically closed. But it is easy to see that the theorem holds for
general $k$ if and only if it holds for its algebraic closure $\Bar{k}$. Thus
there is no loss of generality in the assumption $k=\Bar{k}$.

In the sequel, we call statement (1) of Theorem~\ref{T:equiv} 
\emph{Kemper-Malle conjecture}. Apart from irreducible groups, G. Kemper and 
G. Malle proved it for groups $G$ acting on $2$-dimensional vector space $V$
(for $2$-dimensional groups in characteristic $p>3$ it was first proved by
H. Nakajima \cite{Nakajima}), for some indecomposable groups, and showed that it 
suffices to prove the conjecture in general indecomposable case (\cite{KM}). We 
are motivated by the problem of classifying the isolated quotient singularities 
and find Kemper-Malle conjecture to be a key for the modular case. Indeed, as 
follows from Theorem~\ref{T:equiv}, the meaning of the conjecture is that if it 
holds, then \emph{essentially modular isolated (linear) quotient singularities 
do not exist}, that is those that exist are in fact isomorphic to nonmodular 
singularities. On the other hand, if statement (2) of Theorem~\ref{T:equiv} 
could be proven with methods of algebraic geometry, this would imply the 
conjecture. We succeeded only in dimension $3$ and odd characteristic $p>2$. 
Our method relies on the classification of $2$-dimensional groups 
generated by transvections and does not generalize to higher dimensions. Our 
main result is the following
\begin{theorem}\label{T:main}
Let $V$ be a $3$-dimensional vector space over an algebraically closed field
of characteristic $p>2$. Let $G$ be a finite subgroup of $GL(V)$ generated
by pseudoreflections. Then if the quotient $V/G$ is singular, the singularity
is not isolated, i.e., Kemper-Malle conjecture holds for such groups.
\end{theorem}
We say that two singularities $P\in X$ and $Q\in Y$ are \emph{formally 
isomorphic}, if there exists an isomorphism of complete local rings
$\widehat{\mathcal{O}_{X,P}}\simeq \widehat{\mathcal{O}_{Y,Q}}$.
\begin{corollary}
If $\dim V=3$, $p>2$, and $G$ is a finite subgroup of $GL(V)$ such that 
$0\in V/G$ is an isolated singularity, then $0\in V/G$ is formally isomorphic
to one of nonmodular isolated quotient singularities in Zassenhaus-Vincent-Wolf
classification.
\end{corollary}
In fact, our result is a bit stronger than Theorem~\ref{T:main} and is
also applicable to some groups in even characteristic; see 
Section~\ref{S:proof} for the precise formulation.

The paper is organized as follows. In Section~\ref{S:prelim} we collect
some material on quotient singularities used in the consequent sections. 
In Section~\ref{S:KMandIQS} we prove Theorem~\ref{T:equiv}. We also
show that Kemper-Malle conjecture holds for groups $G$ possessing a 
$1$-dimensional invariant subspace, in particular, for Abelian $G$.
Section~\ref{S:proof} is devoted to the proof of Theorem~\ref{T:main}.

The author is pleased to thank Professors G. Kemper and V.~L. Popov for
providing him with useful information on modular group actions.

\section{Preliminaries}\label{S:prelim}

The results of this section have already appeared in \cite{Kemper}. For the reader's 
convenience, we state them in our notation and with complete proofs. If not stated 
otherwise, $k$ denotes in this section a field of characteristic $p\geq 0$, not 
necessarily algebraically closed.

\begin{lemma}\label{L:regofquot}
Let a finite group $G$ act linearly on the polynomial ring 
$R=k[x_1,\dots,x_n]$. Then the ring of invariants $R^G$ is polynomial if
and only if $R^G$ is regular at the maximal ideal $\mathfrak{m}\cap R^G$,
where $\mathfrak{m}=(x_1,\dots,x_n)$.
\end{lemma}
\begin{proof}
Only the sufficiency needs a proof. Suppose that $R^G$ is regular at 
$\mathfrak{m}\cap R^G$, but $R^G$ is not polynomial. But $R^G$ is a finitely
generated $k$-algebra (see, e.g., \cite{Benson}, Theorem~1.3.1), so let $f_1$, 
$\dots$, $f_m$ be a minimal set of generators of $R^G$. Let $g_1(y_1,\dots,y_m)$, 
$\dots$, $g_r(y_1,\dots,y_m)$ be a generating set of all relations between $f_1$,
$\dots$, $f_m$. Note that all the polynomials $f_i$, $i=1,\dots,m$, can be 
chosen homogeneous of positive degree, whereas $g_j$, $j=1,\dots,r$,  weighted 
homogeneous with 
$$\text{weight of }y_i=\deg f_i.$$
It follows that $g_j(0,\dots,0)=0$ for all $j=1,\dots,r$. Moreover, all the
monomials of $g_j$ have degree $>1$, because otherwise the set of generators 
$f_1$, $\dots$, $f_m$ would not be minimal. Thus by Jacobian criterion the 
ring 
$$k[y_1,\dots,k_m]\simeq R^G$$
is not regular at $0$, a contradiction.
\end{proof}

Now consider two algebras $A$ and $B$ without zero divisors over a 
field $k$, where $A$ is a subalgebra of $B$. Let $\mathfrak{m}\subset A$ and
$\mathfrak{n}\subset B$ be maximal ideals such that $\mathfrak{n}\cap A=
\mathfrak{m}$. Denote by $j\colon A\to B$ the inclusion, by $\widehat{A}$ and
$\widehat{B}$ the formal completions of $A$ at $\mathfrak{m}$ and $B$ at
$\mathfrak{n}$ respectively, and by $\widehat{j}\colon \widehat{A}\to 
\widehat{B}$ the map induced by $j$.

\begin{lemma}\label{L:alglem}
Assume that the following conditions are satisfied: 
\begin{enumerate}
\item $A$ and $B$ are Noetherian;
\item $A/\mathfrak{m}=B/\mathfrak{n}=k$;
\item $B$ is unramified at $\mathfrak{n}$ over $A$, that is
$\mathfrak{m}B_\mathfrak{n}=\mathfrak{n}B_\mathfrak{n}$,
where $B_\mathfrak{n}$ is the localization of $B$ at $\mathfrak{n}$.
\end{enumerate}
Then $\widehat{j}$ is an isomorphism.
\end{lemma}
\begin{proof}
Let
$$j_n\colon \mathfrak{m}^n/\mathfrak{m}^{n+1}\to 
\mathfrak{n}^n/\mathfrak{n}^{n+1},\: n\geq 0,$$
be the natural map induced by $j$. By \cite{AM}, Lemma~10.23, it is enough to
show that $j_n$ is an isomorphism for each $n\geq 0$. Note that
$$\mathfrak{n}^n/\mathfrak{n}^{n+1}\simeq 
(\mathfrak{n}B_\mathfrak{n})^n / (\mathfrak{n}B_\mathfrak{n})^{n+1},$$
thus surjectivity of $j_n$ follows easily from condition (3). Let us
prove injectivity of $j_n$.

Assume on the contrary that $j_n$ is not injective for some $n\geq 0$. This
means that there is an element $a\in\mathfrak{m}^n$ such that $a\notin
\mathfrak{m}^{n+1}$, but $a\in\mathfrak{n}^{n+1}$. Considering $a$ as an
element of $B_\mathfrak{n}$ and using again condition (3), we can write $a$
as
$$a=\sum b_i a_i,$$
where $b_i\in B_\mathfrak{n}$, $a_i\in\mathfrak{m}^{n+1}$. Now use condition 
(2) and the fact that the field $k$ is contained in $A$ and $B$ to write 
$b_i=b_{i}^{0}+b'_{i}$, $b_{i}^{0}\in k$, $b'_{i}\in
\mathfrak{n}B_\mathfrak{n}$. This allows to rewrite $a$ as
$$a=\sum b_{i}^{0}a_i + b',$$
where, on the one hand, $b'$ is an element of $A$ not contained in
$\mathfrak{m}^{n+1}$ (otherwise we would have $a\in\mathfrak{m}^{n+1}$), on the
other hand, $b'$ is contained in $(\mathfrak{n}B_\mathfrak{n})^{n+2}\cap B =
\mathfrak{n}^{n+2}$. Moreover, $b'\ne 0$ because otherwise $a$ would not be 
contained in $\mathfrak{n}^{n+1}$. Applying the same argument to $b'$ we get 
$0\ne b''\in\mathfrak{n}^{n+3}$, $b''\in\mathfrak{m}^n$, $b''\notin
\mathfrak{m}^{n+1}$. It follows that $\mathfrak{m}^n\setminus
\mathfrak{m}^{n+1}$ contains nonzero elements of $\mathfrak{n}^N$ for
arbitrary large $N$.

The ideals $\mathfrak{n}^N$ are $k$-vector subspaces of $B$. Condition (1)
implies that $V=\mathfrak{m}^n/\mathfrak{m}^{n+1}$ is a finite dimensional 
$k$-vector space. Thus 
$$V_N=(\mathfrak{n}^N\cap \mathfrak{m}^n)/\mathfrak{m}^{n+1},\: N>n+1,$$
is a descending sequence of subspaces of $V$. This sequence stabilizes on some
subspace $W$ of $V$. We saw above that each member of the sequence has a 
nonzero element, thus $W\ne 0$. On the other hand, by Krull's theorem
$\cap_{n\geq 0}\mathfrak{n}^n=(0)$, so we must have $W=0$. This contradiction
proves injectivity of $j_n$ and hence the lemma.
\end{proof}

Algebraic Lemma~\ref{L:alglem} can be applied to quotient singularities in
the following way. Let $G$ be a finite group that acts linearly on a vector
space $V$. Let $P\in V$ be a (closed) point and $H$ the stabilizer of
$P$ in $G$. Consider the quotients $V/H=\spec S(V^*)^H$ and $V/G=
\spec S(V^*)^G$. Let $Q\in V/H$ and $R\in V/G$ be the images of $P$ under the
natural projections $p_H\colon V\to V/H$ and $p_G\colon V\to V/G$. In general 
$H$ is not normal in $G$, so there is no natural group action on $V/H$, but 
anyway there is a morphism $\varphi\colon V/H\to V/G$ that makes the following 
diagram commutative:
$$\begin{xymatrix}{
&V \ar[dd]_{p_G} \ar[rd]^{p_H} & \\
& &V/H \ar[ld]^{\varphi} \\
&V/G &
} 
\end{xymatrix}
$$
Let $\widehat{\mathcal{O}}_{V/H,Q}$ and $\widehat{\mathcal{O}}_{V/G,R}$ be the
complete local rings of the points $Q\in V/H$ and $R\in V/G$ respectively.

\begin{lemma}\label{L:isomofquots}
Assume that the ground field $k$ is infinite. Then the map
$$\widehat{\varphi}\colon
\widehat{\mathcal{O}}_{V/H,Q}\to \widehat{\mathcal{O}}_{V/G,R}$$
induced by $\varphi$ is an isomorphism.
\end{lemma}
\begin{proof}
Let $\mathfrak{m}_P$ be the maximal ideal of the point $P\in V$, 
$\mathfrak{m}_Q=\mathfrak{m}_P \cap S(V^*)^H$, $\mathfrak{m}_R=\mathfrak{m}_P
\cap S(V^*)^G$. Then apply Lemma~\ref{L:alglem} to the algebras $A=S(V^*)^G$,
$B=S(V^*)^H$, and the ideals $\mathfrak{m}_R$ and $\mathfrak{m}_Q$. We have to
check only condition (3) of Lemma~\ref{L:alglem}, i.e., that $B$ is
unramified at $\mathfrak{m}_Q$ over $A$.

First let us show that the ideal $\mathfrak{m}_Q$ is generated by polynomials
$f_1$, $\dots$, $f_n\in S(V^*)$ such that for all $i$, $1\leq i\leq n$,
$f_i(P)=0$, but for all $g\in G$, $g\notin H$, $f_i(gP)\ne 0$. Indeed, choose
a linear function $l$ on $V$ such that $l(P)=0$ but $l(gP)\ne 0$ for all
$g\in G$, $g\notin H$. Such choice is possible since $k$ is infinite. Consider
an invariant
$$L(v)=\prod_{h\in H} l(hv)$$
of the group $H$. Suppose now that $f'_1$, $\dots$, $f'_{n-1}$ is any 
generating set of the ideal $\mathfrak{m}_Q$. Then one may choose $c_1$, 
$\dots$, $c_{n-1}\in k$ so that the system
$$f_1=f'_1+c_1L,\:\dots,\:f_{n-1}=f'_{n-1}+c_{n-1}L,\: f_n=L$$
of generators of $\mathfrak{m}_Q$ has the desired property.

Now for each $i$, $1\leq i\leq n$, consider
$$g_i(v)=N_{H}^{G}(f_i)(v)=\prod_g f_i(gv),$$
where $g$ runs over a system of representatives of all right classes of $G$
modulo $H$. By construction $g_i$ is an invariant of $G$ and belongs to
$\mathfrak{m}_R$. Since $f_i$ is an invariant of $H$, 
$$\frac{g_i}{f_i}=\prod_{g\notin H} f_i(gv),$$
where the product is taken over all representatives of nontrivial right
classes of $G$ modulo $H$, is also an invariant of $H$. Moreover, 
$g_i/f_i(Q)\ne 0$, thus it is a unit in the local ring $\mathcal{O}_{V/H,Q}$.
It follows that $g_1$, $\dots$, $g_n$ generate the ideal 
$\mathfrak{m}_Q\cdot\mathcal{O}_{V/H,Q}$.
\end{proof}

\begin{lemma}\label{L:fixator}
Let a finite group $G$ act linearly on a finite dimensional vector space $V$.
Suppose that $W$ is a subspace of $V$ which is pointwise fixed by $G$. Let
$P_1\in W$ and $P_2\in W$ be two (closed) points, and $Q_1\in V/G$, $Q_2\in V/G$
their images under the natural projection $p\colon V\to V/G$. Then the local 
rings of $Q_1$ and $Q_2$ are isomorphic:
$$\mathcal{O}_{V/G,Q_1}\simeq\mathcal{O}_{V/G,Q_2}.$$
\end{lemma}
\begin{proof}
Let $v$ be a vector of $V$ that joins $P_1$ to $P_2$: $P_1+v=P_2$. The 
translation by vector $v$ is an automorphism of $V$ as a scheme. Since $v\in W$, 
this  translation can be pushed forward along $p$ to an automorphism of $V/G$ 
that maps the point $Q_1$ to $Q_2$. The lemma follows.
\end{proof}

\section{Kemper-Malle conjecture and isolated quotient singularities}
\label{S:KMandIQS}
In this section the field $k$ will be assumed algebraically closed. Let us 
prove Theorem~\ref{T:equiv}. First consider the implication $(2)\Rightarrow(1)$.
Let $G$ be a subgroup of $GL(V)$ generated by pseudoreflections. By 
Lemma~\ref{L:regofquot}, the ring $S(V^*)^G$ is polynomial if and only if $V/G$ 
is nonsingular at the image of the origin. But then $V/G$ is nonsingular 
everywhere. Indeed, the singular set is closed, and $V/G$ can be singular only 
at the image of a linear subspace of $V$. Then our theorem follows from
Chevalley-Shephard-Todd Theorem~\ref{T:CST} and Lemma~\ref{L:isomofquots}.

If $W$ is a subspace of $V$, denote by $\fix(W)$ the pointwise stabilizer
of $W$ in $G$. In the sequel we refer to $\fix(W)$ as the \emph{fixator} of
the subspace $W$. Suppose that $S(V^*)^{\fix(W)}$ is polynomial for every
nontrivial subspace $W$ of $V$. Then by Lemma~\ref{L:isomofquots} $V/G$ is
nonsingular everywhere except possibly the image of the origin. But since $G$
is generated by pseudoreflections, $V/G$ is nonsingular by (2).

Now let us prove the implication $(1)\Rightarrow(2)$. Let $G$ be a finite
subgroup of $GL(V)$ such that the singularity $V/G$ is isolated, and $H$ the
subgroup of $G$ generated by pseudoreflections. Then by 
Lemma~\ref{L:isomofquots} and Chevalley-Shephard-Todd Theorem the fixator
$\fix(W)$ of every nontrivial subspace of $V$ is generated by pseudoreflections
and, moreover, $V/\fix(W)$ is nonsingular. Note also that $\fix(W)$ is contained
in $H$. Then, by (1), the quotient $V/H$ is nonsingular. Further, let $g$ be an
element of $G$ of order $p^r$, $r>0$, $p=\charact k$. Such an element
necessarily fixes a subspace $W$ of positive dimension. Thus $g\in\fix(W)$;
in particular, $g$ is contained in $H$. It follows that $p\nmid |G/H|$.
The nonmodular group $G/H$ acts naturally on $V/H$ and
$$V/G\simeq (V/H)/(G/H).$$
Theorem~\ref{T:equiv} is proven.

\begin{lemma}\label{L:dim1}
Kemper-Malle conjecture holds for groups $G$ possessing a $1$-dimensional
invariant subspace.
\end{lemma}
\begin{proof}
In view of Theorem~\ref{T:CST}, we have to show that if $G$ is generated
by pseudoreflections, has a $1$-dimensional invariant subspace, and $V/G$ is
singular, then the singularity $V/G$ is nonisolated. Let $W$ be the 
$1$-dimensional invariant subspace of $G$. Note that in this case the
fixator $\fix(W)$ is normal in $G$. If $V/\fix(W)$ is singular, then
by Lemma~\ref{L:fixator} it is nonisolated. Thus by Lemma~\ref{L:isomofquots}
$V/G$ is also nonisolated. 

Suppose that $V/\fix(W)$ is nonsingular. If $g$ is an element of $G$ of
order $p^r$, $r>0$, then the restriction of $g$ to the subspace $W$ is trivial.
Thus $g\in \fix(W)$. It follows that the quotient group $G/\fix(W)$ is
nonmodular. Its action on $V/\fix(W)$ can be locally formally linearized
(see, e.g., \cite{St}, Lemma~2.3), and, since $G$ is generated by 
pseudoreflections, the linearization of $G/\fix(W)$ is also generated by 
pseudoreflections. Then it follows from Chevalley-Shephard-Todd Theorem that
$$V/G\simeq (V/\fix(W))/(G/\fix(W))$$
is nonsingular.
\end{proof}

\begin{corollary}
Kemper-Malle conjecture holds for all Abelian groups generated by
pseudoreflections.
\end{corollary}
\begin{proof}
Indeed, a linear Abelian group always has a $1$-dimensional invariant subspace.
\end{proof}

\begin{corollary}
Let $G<GL(V)$ be a finite Abelian group such that the singularity $V/G$ is
isolated. Then $V/G$ is formally isomorphic to a nonmodular cyclic singularity.
\end{corollary}
\begin{proof}
We can get rid of pseudoreflections and assume that the group $G$ is nonmodular.
Then the statement follows from Zassenhaus-Vincent-Wolf classification of
isolated quotient singularities, see \cite{Wolf} or \cite{St}.
\end{proof}

Examples illustrating statement (2) of Theorem~\ref{T:equiv} are easy to
construct.
\begin{example}
Let $k$ be a field of characteristic $3$ and $V=k^2$. Let $G$ be a subgroup of
$GL(2,k)$ generated by
$$s=\begin{pmatrix}
2 & 0 \\
0 & 2
\end{pmatrix} \text{ and }
t=\begin{pmatrix}
1 & 1 \\
0 & 1
\end{pmatrix}.
$$
Then $G$ is isomorphic to the direct sum $\Z/2\oplus\Z/3$, and $t$ is a
pseudoreflection (transvection, see Section~\ref{S:proof}) generating a subgroup 
$H$ isomorphic to $\Z/3$. Basis invariants of $H$ are
$$f_1=x(x+y)(x+2y) \text{ and } f_2=y,$$
where $x,y$ is a basis of $V^*$. Thus $V/H$ is nonsingular. $G/H\simeq\Z/2$ acts
on $f_1$ and $f_2$ via $f_1\mapsto -f_1$, $f_2\mapsto -f_2$, thus the singularity
$V/G$ is isolated and isomorphic to a quadratic cone.
\end{example}

\section{Proof of Theorem~\ref{T:main}}\label{S:proof}

Now we are going to prove our main result, Theorem~\ref{T:main}. It will be
a consequence of more general Theorem~\ref{T:main_gen}. First let us recall
some terminology. A pseudoreflection $g\colon V\to V$ is called a 
\emph{transvection}, if $g$ has the only eigenvalue $1$. A transvection
necessarily has order $p$ equal to characteristic of the base field $k$. Note
that if $\dim V=2$, then any element of $GL(V)$ of order $p^r$, $r\geq 1$,
has in fact order $p$ and is a transvection.

\begin{theorem}\label{T:main_gen}
Let $V$ be a $3$-dimensional vector space over an algebraically closed field
of any characteristic $p$. Let $G$ be a finite subgroup of $GL(V)$ generated
by pseudoreflections. Denote by $G_p$ the normal subgroup of $G$ generated by 
all elements of order $p^r$, $r\geq 1$. Assume that $G_p$ is either
\begin{enumerate}
\item irreducible on $V$, or
\item has a $1$-dimensional invariant subspace $U$, or
\item has a $2$-dimensional invariant subspace $W$ and the restriction of $G_p$
to $W$ is generated by two noncommuting transvections (and thus is irreducible).
\end{enumerate}
Then Kemper-Malle conjecture holds for $G$, i.e., if $V/G$ is singular, then
the singularity is not isolated. Moreover, if $G$ satisfies condition (3) or
condition (2) plus the induced action of $G_p$ on $V/U$ is generated by two
noncommuting transvections, then $V/G$ is nonsingular.
\end{theorem}

\begin{remark}\label{R:tgroupsindim2}
To see how Theorem~\ref{T:main} follows from Theorem~\ref{T:main_gen}, suppose
that $p$ is odd. In case (3) of Theorem~\ref{T:main_gen}, denote by $H$ the
restriction of the group $G_p$ to $W$. Since $\dim W=2$, $H$ is an irreducible
group generated by transvections. It follows from general representation
theory that $H$ is defined over a finite extension of the prime subfield of $k$
(\cite{CR}, Chapter XII). Classification of such groups is known since
the beginning of the twentieth century (see, e.g., \cite{Mitchell}, \S~1); it
implies that $H$ is conjugate in $GL(W)$ to the group $SL(2,q)$, $q=p^n$ -- the
group of $2\times 2$ matrices of determinant $1$ with entries in the Galois 
field $\Fq$, -- or, in characteristic $p=3$, $H$ may also be conjugate to the
binary icosahedral group $I^*\simeq SL(2,5)$ in a $3$-modular representation. In 
the last case $H$ is conjugate to the subgroup of $SL(2,9)$ generated by two
transvections
$$
t= \begin{pmatrix}
1 & 1 \\
0 & 1
\end{pmatrix} \text{ and }
s= \begin{pmatrix}
1 & 0 \\
\lambda & 1
\end{pmatrix},
$$
where $\lambda^2=-1$. Each group $SL(2,q)$ for odd $q$ is also generated by two
appropriate noncommuting transvections. In characteristic $2$, each group
generated by two noncommuting transvections is conjugate to an imprimitive group 
generated by
$$
t= \begin{pmatrix}
0 & 1 \\
1 & 0
\end{pmatrix} \text{ and }
s= \begin{pmatrix}
x & 0 \\
0 & x^{-1}
\end{pmatrix},
$$
where $x\ne 0,1$ is an element of the field $\mathbb{F}_{2^n}$. Obviously, this 
group is isomorphic to the dihedral group $D_n$. The group $SL(2,2^n)$, $n>1$, 
is not generated by two transvections, whereas over an algebraically closed
field the group $SL(2,2)$ is conjugate to an imprimitive group described above.
\end{remark}

\begin{proof}
The case of irreducible groups $G$ is proved by G. Kemper and G. Malle in
\cite{KM}. The proof in case (2) follows from Lemma~\ref{L:dim1}. So we 
concentrate on the proof of case (3).

Note that if we show that the quotient $V/G_p$ is nonsingular, then a 
nonmodular group $G/G_p$ generated by pseudoreflections acts naturally on the
nonsingular variety $V/G_p$. Such an action is locally formally linearizable,
thus the quotient $V/G\simeq (V/G_p)/(G/G_p)$ is also nonsingular. Therefore
it is enough to prove our theorem in the case $G=G_p$.

Denote by $N$ the kernel of the restriction map $G\to H$, so that we have an
extension
\begin{equation}\label{E:extension}
1\to N\to G\to H\to 1.
\end{equation}
The group $N$ is Abelian and consists of the matrices of the form
$$\begin{pmatrix}
1 & 0 & a \\
0 & 1 & b \\
0 & 0 & 1
\end{pmatrix},
$$
where $a$, $b\in k$ and the basis is chosen so that the invariant subspace $W$
is generated by the first two vectors.

\begin{lemma}\label{L:ting}
In the conditions of case (3) of Theorem~\ref{T:main_gen}, the group $G$ 
necessarily contains transvections that restrict to (nontrivial) transvections of
the group $H$.
\end{lemma}
\begin{proof}
According to the classification of the irreducible groups generated by 
transvections, the group $H$ is one of the following: $SL(2,q)$, $q$ is odd;
$I^*\simeq SL(2,5)$ in the $3$-modular representation described in 
Remark~\ref{R:tgroupsindim2}; the imprimitive $2$-modular group also described
in Remark~\ref{R:tgroupsindim2}. In fact, our proof works also for
the groups $H=SL(2,2^n)$, $n>1$, so we shall study also this case. Let 
us consider these possibilities one by one. In each case we assume that $G$ does 
not contain transvections with nontrivial image in $H$, and come to a 
contradiction.

\emph{Case 1}: $H=SL(2,q)$, $q$ is odd. Then $H$ contains matrices
$$
t= \begin{pmatrix}
1 & 1 \\
0 & 1
\end{pmatrix} \text{ and }
s= \begin{pmatrix}
1 & 0 \\
1 & 1
\end{pmatrix}.
$$
Let $\tilde{t}$ and $\tilde{s}$ be some lifts of $t$ and $s$ to $G$ respectively.
We assume that $\tilde{t}$ and $\tilde{s}$ are not transvections. Then one can
choose a basis in $V$ so that
$$
\tilde{t}= \begin{pmatrix}
1 & 1 & 0 \\
0 & 1 & 1 \\
0 & 0 & 1
\end{pmatrix},\quad
\tilde{s}= \begin{pmatrix}
1 & 0 & \mu \\
1 & 1 & 0 \\
0 & 0 & 1
\end{pmatrix},
$$
where $\mu\ne 0$ is an element of $k$. One has
$$
\tilde{t}^{-1}\tilde{s}\tilde{t}= \begin{pmatrix}
0 & -1 & \mu \\
1 & 2 & 0 \\
0 & 0 & 1
\end{pmatrix},\quad
\tilde{s}^{-1}\tilde{t}\tilde{s}= \begin{pmatrix}
2 & 1 & 0 \\
-1 & 0 & 1 \\
0 & 0 & 1
\end{pmatrix},
$$
$$
u=\tilde{t}^{-1}\tilde{s}\tilde{t}\tilde{s}^{-1}\tilde{t}\tilde{s}=
\begin{pmatrix}
1 & 0 & \mu-1 \\
0 & 1 & 2 \\
0 & 0 & 1
\end{pmatrix} \in N.
$$
Then
$$u^{-1}\tilde{s}=\begin{pmatrix}
1 & 0 & 1 \\
1 & 1 & -2 \\
0 & 0 & 1
\end{pmatrix}
$$
It follows that we could assume from the beginning that $\mu\in \Fq$.
Conjugating $u$ with appropriate elements of $G$ one gets elements of $N$ with
arbitrary vectors $(a,b,1)^T$, $a$, $b\in \Fq$ in place of the third column. In
particular, there is an element $u_1$ in $N$ with the third column $(0,-1,1)^T$. 
The product  $u_{1}\tilde{t}$ is a transvection that restricts to $t\in H$.

\emph{Case 2}: $H=I^*<SL(2,9)$. We represent the field $\mathbb{F}_9$ as the 
decomposition field of the polynomial
$$x^2+x+2$$
over $\mathbb{F}_3$. Then we can choose the transvections generating $H$ to be
$$
t= \begin{pmatrix}
1 & 1 \\
0 & 1
\end{pmatrix} \text{ and }
s= \begin{pmatrix}
1 & 0 \\
x+2 & 1
\end{pmatrix}.
$$
Denote again by $\tilde{t}$ and $\tilde{s}$ some lifts of $t$ and $s$ to $G$
respectively. If $\tilde{t}$ and $\tilde{s}$ are not transvections, in a 
suitable basis they have matrices
$$
\tilde{t}= \begin{pmatrix}
1 & 1 & 0 \\
0 & 1 & 1 \\
0 & 0 & 1
\end{pmatrix},\quad
\tilde{s}= \begin{pmatrix}
1 & 0 & \mu \\
x+2 & 1 & 0 \\
0 & 0 & 1
\end{pmatrix},
$$
where $\mu\ne 0$ is an element of $k$. By a routine calculation one checks that
$$(\tilde{t}\tilde{s})^5= \begin{pmatrix}
2 & 0 & 2x+1 \\
0 & 2 & \mu+2 \\
0 & 0 & 1
\end{pmatrix}, \quad
(\tilde{s}\tilde{t})^5= \begin{pmatrix}
2 & 0 & 2x+2\mu+1 \\
0 & 2 & \mu \\
0 & 0 & 1
\end{pmatrix},
$$
$$
u=(\tilde{t}\tilde{s})^5 (\tilde{s}\tilde{t})^5= \begin{pmatrix}
1 & 0 & \mu \\
0 & 1 & 2 \\
0 & 0 & 1
\end{pmatrix} \in N.
$$
It follows that
$$\tilde{u}=u\tilde{s}= \begin{pmatrix}
1 & 0 & 0 \\
x+2 & 1 & 1 \\
0 & 0 & 1
\end{pmatrix} \in G.
$$
Then one has
$$(\tilde{t}\tilde{u})^5= \begin{pmatrix}
2 & 0 & 2x+1 \\
0 & 2 & 2 \\
0 & 0 & 1
\end{pmatrix}, \quad
(\tilde{u}\tilde{t})^5= \begin{pmatrix}
2 & 0 & x+2 \\
0 & 2 & 0 \\
0 & 0 & 1
\end{pmatrix},
$$
and
$$
u_1=(\tilde{t}\tilde{u})^5 (\tilde{u}\tilde{t})^5= \begin{pmatrix}
1 & 0 & 0 \\
0 & 1 & 2 \\
0 & 0 & 1
\end{pmatrix} \in N.
$$
Then $(u_{1}^{-1} u)^{-1} \tilde{s}$ is a transvection that restricts to 
$s\in H$.

\emph{Case 3}: $H$ is imprimitive, the characteristic of $k$ is $2$. The
generators $t$ and $s$ of $H$ are defined in Remark~\ref{R:tgroupsindim2}. Their
lifts to $G$ can be chosen to be
$$\tilde{t}= \begin{pmatrix}
0 & 1 & \mu \\
1 & 0 & \mu^{-1} \\
0 & 0 & 1
\end{pmatrix} \text{ and }
\tilde{s}= \begin{pmatrix}
x & 0 & 0 \\
0 & x^{-1} & 0 \\
0 & 0 & 1
\end{pmatrix}
$$
respectively; here $\mu\ne 0,1$ is an element of $k$. One has
$$\tilde{t}^2= \begin{pmatrix}
1 & 0 & \mu+\mu^{-1} \\
0 & 1 & \mu+\mu^{-1} \\
0 & 0 & 1
\end{pmatrix},\:
(\tilde{s}\tilde{t})^2= \begin{pmatrix}
1 & 0 & \mu^{-1}+x\mu \\
0 & 1 & \mu+x^{-1}\mu^{-1} \\
0 & 0 & 1
\end{pmatrix} \in N.
$$
Together with $\tilde{t}^2$, $(\tilde{s}\tilde{t})^2$ the group $N$ contains also 
all their conjugates and products. One easily deduces that $N$ contains all 
matrices with the third column 
$$(f(x)(x+1)\mu,f(x^{-1})(x^{-1}+1)\mu^{-1},1)^T$$
for all polynomials $f$ with coefficients in $\mathbb{F}_2$. If the minimal
polynomial $g$ of $x$ over $\mathbb{F}_2$ is reciprocal, that is $g(x^{-1})=0$,
then if $h(x)(x+1)=1$ for some polynomial $h$, then also $h(x^{-1})(x^{-1}+1)=1$.
It follows that $N$ contains a matrix $u$ with the third column 
$(\mu,\mu^{-1},1)^T$. Then $u\tilde{t}$ is a transvection that restricts to
$t\in H$. If $g$ is not reciprocal, then, taking $f=g$, we see that $N$ contains
a matrix with the third column 
$$(0,g(x^{-1})(x^{-1}+1)\mu^{-1},1)^T.$$
Its conjugates and their combinations produce also a matrix with the third 
column $(0,\mu^{-1},1)^T$. Taking $f$ to be the minimal polynomial for $x^{-1}$, 
we find also a matrix in $N$ with the third column $(\mu,0,1)$. Then it again 
easily follows that $G$ has a transvection that restricts to $\tilde{t}$.

\emph{Case 4}: $H=SL(2,2^n)$, $n>1$. Since $SL(2,2)$ is a subgroup of each 
$SL(2,2^n)$, it is enough to prove the lemma for $H=SL(2,2)$. But this is
already proved above since $SL(2,2)$ is conjugate to an imprimitive group.
\end{proof}

\begin{lemma}
The group $G$ contains transvections that restrict to generators of the group
$H$. This holds also for $H=SL(2,2^n)$, $n>1$.
\end{lemma}
\begin{proof}
The group $H=SL(2,2^n)$, as well as imprimitive groups $H$ and $H=I^*<SL(2,9)$, 
have only one conjugacy class of transvections. Thus for these groups the lemma 
follows directly from Lemma~\ref{L:ting}. The groups $H=SL(2,q)$, $q$ is odd, 
have two conjugacy classes of transvections. The transvections
$$
t'= \begin{pmatrix}
1 & a \\
0 & 1
\end{pmatrix} \text{ and }
t''= \begin{pmatrix}
1 & b \\
0 & 1
\end{pmatrix},
$$
$a$, $b\in\Fq$, are conjugate if and only if $a$ and $b$ simultaneously are
or are not squares in $\Fq$. The two noncommuting transvections $t$ and $s$ 
generating $SL(2,q)$ can be chosen in the form
$$
t= \begin{pmatrix}
1 & 1 \\
0 & 1
\end{pmatrix}, \quad
s= \begin{pmatrix}
1 & 0 \\
\lambda & 1
\end{pmatrix},
$$
where $\lambda$ must be an element of $\Fq$ not belonging to a smaller field, 
and $\lambda^2\ne -1$ if $q=9$. We have already seen in the proof of 
Lemma~\ref{L:ting} that $t$ is a restriction of a transvection from $G$.
Clearly $\lambda$ in $s$ can be chosen to be or not to be a square in $\Fq$, so 
that $t$ and $s$ are conjugate.
\end{proof}

\begin{lemma}\label{L:splitting}
Extension~\eqref{E:extension} is a semidirect product, and the group
$H$ can be embedded in $G$ so that it acts on $V$ via a decomposable 
representation with two invariant subspaces of dimensions $1$ and $2$.
\end{lemma}
\begin{proof}
Lift two generating transvections $t$ and $s$ of $H$ to transvections $\tilde{t}$
and $\tilde{s}$ of $G$ and consider the subgroup that they generate in $G$.
The planes fixed by $\tilde{t}$ and $\tilde{s}$ intersect at a line not
contained in the invariant subspace $W$. This gives the desired splitting. 
\end{proof}

\begin{remark}
At least for $H=SL(2,q)$, $q$ is odd, Lemma~\ref{L:splitting} could also be 
proved with a help of known results on vanishing of the first and the second 
group cohomology of $SL(2,q)$ with coefficients in the natural module 
(\cite{CPS}, \cite{Chih-Han}), or with a help of the complete reducibility of 
low dimensional modules over $SL(2,q)$ (\cite{Guralnik}). However, we would have 
to study in detail the structure of $N$ as a module over $H$, and consider 
separately the cases $H=I^*$ and $H$ is imprimitive. This is the reason why we 
have preferred the elementary uniform proof presented above.
\end{remark}

From now on, we fix some splitting of \eqref{E:extension} and consider $H$ as a
subgroup of $G$. Next we are going to study the quotient $V/G$ in two steps:
first take the quotient $V/N$ and then consider the induced action of $H$ on
$V/N$. By \cite{CW}, Theorem~3.9.2, the invariant ring of $N$ is polynomial.

\begin{lemma}
The induced action of $H$ on $V/N\simeq k^3$ is linear and decomposable with
invariant subspaces of dimensions $1$ and $2$.
\end{lemma}
\begin{proof}
Fix a basis for $V$ in which the elements of $H$ are represented by block
matrices of the form
$$\begin{pmatrix}
a & b & 0 \\
c & d & 0 \\
0 & 0 & 1
\end{pmatrix}.
$$
Let $x$, $y$, $z$ be the dual basis of $V^*$. Then $N$ acts on $x$, $y$, $z$ via
transformations
$$x\mapsto x+\lambda z,\quad y\mapsto y+\mu z,\quad z\mapsto z,$$
$\lambda$, $\mu\in k$. Clearly, $f_3=z$ is invariant. Let $f_1$ and $f_2$ be
two other invariants that together with $f_3$ generate $S(V^*)^N$. One can
choose $f_1$ and $f_2$ homogeneous and not containing the monomial $z^m$ with a
nonzero coefficient. First let us assume that $f_1$ and $f_2$ have different
degrees, say, $\deg f_2< \deg f_1$. Then $f_2$ must be semiinvariant under the 
induced action of $H$. Indeed, the degree of $f_2$ is preserved, thus for
$h\in H$ the polynomial $h\cdot f_2$ expresses only through $f_2$ and $f_3$,
but also $h$ acts only on $x$ and $y$, thus $h\cdot f_2$ expresses only through
itself. Moreover, since $H$ is generated by elements of order $p$ but any
element of $\Fq^*$ has multiplicative order coprime to $p$, $f_2$ is in fact
an invariant of $H$. For the same reason if
$$h\cdot f_1 = \lambda f_1 + g(f_2,f_3),$$
then $\lambda=1$ for each $h\in H$. But then the subset of $V/N$ defined by 
equations $f_2=f_3=0$ is pointwise fixed under the action of $H$. The natural 
projection $V\to V/N$ is $H$-equivariant. In turn, this implies that $G$ has an 
invariant line contained in the subspace $W=\{z=0\}$, which contradicts the 
condition $H$ to be irreducible on $W$. 

It follows that $\deg f_1= \deg f_2$. By the same argument as above $h\cdot f_1$ 
and $h\cdot f_2$ are linear combinations of $f_1$ and $f_2$. On $V/N$ the group
$H$ acts also by block matrices, thus the representation is decomposable. 
It should be also possible to prove this lemma by a direct computation of the 
invariants of $N$ and the induced action of $H$ on them.
\end{proof}

Let $V/N=V_1\oplus V_2$ be the decomposition of $V/N$ into a sum of invariant 
$H$-modules, $\dim V_1=2$, $\dim V_2=1$. Since the group $H$ acting on $V$ is 
generated by transvections, the action of $H$ on $V_1$ and $V_2$ is also 
generated by transvections. In particular, the action on $V_2$ is trivial, and
the invariant ring of $H$ acting on $V_1$ is polynomial (see, e.g., \cite{KM},
Proposition~7.1). It follows that the invariant ring of $H$ acting on $V/N$ is
also polynomial, and hence the invariant ring of $G$ acting on $V$ is polynomial.

It remains to show that if $G$ acts on $V$ with a $1$-dimensional invariant
subspace $U$ and the induced action of $G$ on $W=V/U$ is irreducible and
generated by two noncommuting transvections, then $V/G$ is again nonsingular.
This situation is in some sense dual to the one considered above. Denote
$H$ the natural image of $G$ in $GL(W)$ and $N$ the corresponding kernel. We
again have an extension of the form \eqref{E:extension}. In an appropriate basis
$N$ is an Abelian group of the matrices of the form
$$\begin{pmatrix}
1 & a & b \\
0 & 1 & 0 \\
0 & 0 & 1
\end{pmatrix},
$$
where $a$, $b\in k$. An argument, which is similar to the one given above and
which we do not write here, shows that the group $G$ contains transvections that 
map to the generators of $H$. The proof of the following lemma is easy and left 
to the reader.
\begin{lemma}
Any two transvections on a $3$-dimensional vector space have a common 
$2$-dimensional invariant subspace.
\end{lemma}

Now if $\tilde{t}$ and $\tilde{s}$ are transvections of $G$ that map to 
generators of $H$, consider their common $2$-dimensional invariant subspace. 
Clearly it does not contain $U$, thus can be identified with $W$. It follows
that the sequence~\eqref{E:extension} again splits.

In a suitable basis $x$, $y$, $z$ of $V^*$, the group $N$ acts via 
transformations of the form
$$x\mapsto x+ay+bz,\quad y\mapsto y,\quad z\mapsto z.$$
It is not hard to determine basis invariants of $N$. They are 
$$f_1=\prod_{g\in N} g\cdot x,\quad f_2=y,\quad f_3=z,$$
in particular, the invariant ring $S(V^*)^N$ is polynomial. It is also clear in
this case that the induced action of $H$ on $V/N$ is linear and decomposable.
As above, it follows that $(V/N)/H$ and hence $V/G$ is nonsingular. This
finishes the proof of Theorem~\ref{T:main_gen}.
\end{proof}

\end{document}